\documentclass[11pt]{amsart}

\usepackage{amsmath,amssymb,txfonts,amsthm}

\newtheorem{thm}{Theorem}[section]

\newtheorem{prop}[thm]{Proposition}
\newtheorem{lem}[thm]{Lemma}
\newtheorem{cor}[thm]{Corollary}

\theoremstyle{definition}
\newtheorem{exam}[thm]{Example}

\newtheorem{defi}[thm]{Definition}
\newtheorem{rem}[thm]{Remark}

\begin{document}

\author{Morimichi Kawasaki}

\address[Morimichi Kawasaki]{Center for Geometry and Physics, Institute for Basic Science (IBS), Pohang 790-784, Republic of Korea}
\email{kawasaki@ibs.re.kr}
\title{Superheavy Lagrangian immersion in 2-torus}
\maketitle
\begin{abstract}
We show that the union of a meridian and a longitude of the symplectic 2-torus is superheavy in the sense of Entov-Polterovich. By a result of Entov-Polterovich, this implies that the product of this union and the Clifford torus of $\mathbb{C}P^n$ with the Fubini-Study symplectic form cannot be displaced by symplectomorphisms.
\end{abstract}

\section{Introduction}

A diffeomorphism $f$ of a symplectic manifold $(M,\omega)$ is called \textit{symplectomorphism} if $f$ preserves the symplectic form $\omega$.

A subset $U$ is said to be strongly non-displaceable if $f(U)\cap \bar{U}\neq\emptyset$ for any symplectomorphism $f$.

Entov-Polterovich \cite{EP09} investigated strong non-displaceability and defined the superheavy subsets by using the spectral invariants for the pairs of quantum homology class and Hamiltonian functions.

For symplectic manifolds with nontrivial fundamental group, Irie \cite{Ir} studied the spectral invariants for certain Hamiltonian functions which generates Hamiltonian paths with non contractible closed orbits.
By using such Hamiltonian functions, we show the below theorem.

One of examples obtained by using our method is as follows:
\begin{thm}\label{torus}
Let $(\mathbb{T}^2,\omega_{\mathbb{T}^2})=(\mathbb{R}^2/\mathbb{Z}^2,\omega_{\mathbb{T}^2})$ be the 2-torus with the coordinates $(p,q)$ and the symplectic form $dp{\wedge}dq$.
The union $M{\cup}L$ of the meridian curve $M$ and the longitude curve $L$ is a ``$[\mathbb{T}^2]$-superheavy'' subset of $(\mathbb{T}^2,\omega_{\mathbb{T}^2})$.
\end{thm}

$M{\cup}L$ is not displaceable by any homeomorphisms of $\mathbb{T}^2$, however this gives rise to a strongly non-displaceable subset in $\mathbb{C}P^n\times \mathbb{T}^2$.

In fact, Entov-Polterovich showed that for a symplectic manifold $M$, $[M]$-superheavy subsets are strongly non-displaceable, and for the symplectic manifolds $M_1$, $M_2$, the product of superheavy subsets is superheavy in $M_1\times M_2$. Thus we have the following corollary.

\begin{cor}\label{p-torus}
Let $(\mathbb{C}P^n,\omega_{FS})$ be the complex projective space with the Fubini-Study form $\omega_{FS}$ and $C$ be \textit{the Clifford torus} $\{[z_0:{\cdots}:z_n]\in\mathbb{C}P^n; |z_0|={\cdots}=|z_n|\}$ of $\mathbb{C}P^n$.
Then there exists no symplectomorphism $f$ such that $C{\times}(M\cup{L})\cap f(C{\times}(M\cup{L}))=\emptyset$.
\end{cor}

The present paper is organized as follows. We review the definitions in
symplectic geometry and spectral invariants in Section \ref{prel} which are needed to prove Theorem \ref{torus}. We introduce and prove the important proposition (Proposition \ref{spectral main theorem}) to prove Theorem \ref{torus}.
In Section \ref{qs lemma}, we prepare some lemmas which is useful to prove Theorem \ref{torus}.
In Section \ref{exam}, we prove Theorem \ref{torus} and Corollary \ref{p-torus}.
In Section \ref{general}, we give and prove a generalization of Theorem \ref{torus} on surfaces with higher genus (Theorem \ref{torus2}).

\subsection{Acknowledgment}

The author would like to thank Professor Takashi Tsuboi for his helpful guidance and advice.

He also thanks Kei Irie,  Professor Yong-Geun Oh, Professor Kaoru Ono and  Professor Michael Usher for their advice, Sobhan Seyfaddini for recommending me to write generalized version of the main result (Theorem \ref{torus2}), Professor Hiroki Kodama and Professor Emmanuel Opshtein for pointing out many mistakes and confusing notations,  Professor Shigenori Matsumoto for giving him the advice to submit early.

\section{Preliminaries}\label{prel}

\subsection{Definitions}
For a function $F\colon M\to\mathbb{R}$ with compact support, we define the \textit{Hamiltonian vector field} $\operatorname{sgrad}F$ associated with $F$ by
\[\omega(\operatorname{sgrad}F,V)=-dF(V)\text{ for any }V \in \mathcal{X}(M),\]
where $\mathcal{X}(M)$ denotes the set of smooth vector fields on $M$.

For a function $F{\colon}M{\times}[0,1]\to\mathbb{R}$ and $t \in [0,1]$, we define $F_t{\colon}M\to\mathbb{R}$ by $F_t(x)=F(x,t)$. 
We denote by $\{f_t\}$ the isotopy which satisfies $f_0=\mathrm{id}$ and $\frac{d}{dt}f_t(x)=(\operatorname{sgrad}F_t)_{f_t(x)}$. We call this \textit{the Hamiltonian path generated by the Hamiltonian function $F_t$}.
The time-1 map $f_1$ of $\{f_t\}$ is called \textit{the Hamiltonian diffeomorphism generated by the Hamiltonian function $F_t$}.
A diffeomorphism $f$ is called \textit{a Hamiltonian diffeomorphism} if there exists a Hamiltonian function $F_t$ with compact support generating $f$.
A Hamiltonian diffeomorphism is a symplectomorphism.

For a symplectic manifold $(M,{\omega})$, we denote by $\mathrm{Symp}(M,\omega)$, $\mathrm{Ham}(M,\omega)$ and $\widetilde{\mathrm{Ham}}(M,\omega)$, the group of symplectomorphisms, the group of Hamiltonian diffeomorphisms of $(M,{\omega})$ and its universal cover, respectively. Note that $\mathrm{Ham}(M,\omega)$ is a normal subgroup of $\mathrm{Symp}(M,\omega)$.

Let $(M,\omega)$ be a symplectic manifold and $\{f_t\}_{t\in[0,1]}$ and $\{g_t\}_{t\in[0,1]}$ be the Hamiltonian paths generated by Hamiltonian functions $F_t$ and $G_t$, respectively.
Then $\{f_tg_t\}_{t\in[0,1]}$ are generated by the Hamiltonian function
$(F{\sharp}G)(x,t)=F(x,t)+G(f_t^{-1}(x),t)$.

A Hamiltonian function $H$ is called \textit{normalized} if $\int_M H_t(x)\omega^n=0$ for any $t\in [0,1]$. 
\begin{defi}
For functions $F$ and $G$ and a symplectic manifold $(M,\omega)$, \textit{the Poisson bracket} $\{F,G\}\in C^\infty(M)$ is defined by
\[\{F,G\}=\omega(\operatorname{sgrad}G,\operatorname{sgrad}F).\]
\end{defi}

\begin{defi}[\cite{EP09}]

Let $(M,\omega)$ be a symplectic manifold.

A subset $U$ of $M$ is said to be  \textit{displaceable} if there exists a Hamiltonian diffeomorphism $f\in\mathrm{Ham}(M,\omega)$ such that $f(U){\cap}\bar{U}={\emptyset}$.
Ohterwise $U$ is said to be \textit{non-displaceable}.

A subset $U$ of $M$ is said to be \textit{weakly displaceable} if there exist a symplectomorphism $f\in\mathrm{Symp}(M,\omega)$ such that $f(U){\cap}\bar{U}={\emptyset}$.
Otherwise $U$ is said to be \textit{strongly non-displaceable}.
\end{defi}

Since $\mathrm{Ham}(M,\omega)\subset\mathrm{Symp}(M,\omega)$, if $U$ is displaceable, then $U$ is weakly displaceable.

\subsection{Spectral invariants}\label{introduction to spectral invariant}

For a closed connected symplectic manifold $(M,\omega)$, define
\[ {\Gamma}=\frac{\pi_2(M)}{\operatorname{Ker}(c_1)\cap\operatorname{Ker}([\omega])}, \]
where $c_1$ is the first Chern class of $TM$ with an almost complex structure compatible with $\omega$.
The Novikov ring of the closed symplectic manifold $(M,\omega)$ is defined as follows:
\[ {\Lambda}=\left\{\sum_{A\in\Gamma}a_AA;a_A\in\mathbb{C},\#\{A;a_A\neq{0},\int_A{\omega}<R\}<\infty\text{ for any real number }R\right\}.\]
The quantum homology $QH_\ast(M,\omega)$ is a $\Lambda$-module isomorphic to $H_\ast(M;\mathbb{C})\otimes_\mathbb{C}\Lambda$ and $QH_\ast(M,\omega)$ has a ring structure with the multiplication called the \textit{quantum product} (\cite{Oh}).
For each element $a\in{QH}_\ast(M,\omega)$, a functional $c(a,\cdot){\colon}C^\infty(M\times[0,1])\to\mathbb{R}$ is defined in terms of the Hamiltonian Floer theory. The functional $c(a,\cdot)$ is called a \textit{spectral invariant} (\cite{Oh}).
To describe the properties of a spectral invariant, we define the spectrum of a Hamiltonian function as follows:

\begin{defi}[\cite{Oh}]
Let $H \in C^\infty(M\times[0,1])$ be a Hamiltonian function on a closed symplectic manifold $M$. The \textit{spectrum} $\operatorname{Spec}(H)$ of $H$ is defined as follows:
\[\operatorname{Spec}(H)=\left\{\int_0^1H(h_t(x),t)dt-\int_{\mathbb{D}^2}u^\ast\omega\right\}\subset\mathbb{R},\]
where $\{h_t\}_{t\in[0,1]}$ is the Hamiltonian path generated by $H$ and $x \in M$ is a fixed point of $h_1$ whose orbit defined by $\gamma^x(t)=h_t(x)\;(t\in[0,1])$ is a contractible loop and $u{\colon}\mathbb{D}^2{\to}M$ is a disc in $M$ such that $u|_{\partial\mathbb{D}^2}=\gamma^x$.

\end{defi}
We define the \textit{non-degeneracy} of Hamiltonian functions as follows:
\begin{defi}
A Hamiltonian function $H\in C^\infty(M\times[0,1])$ is called \textit{non-degenerate} if for any $x \in M$ which is a fixed point of $h$ whose orbit $\gamma^x$ is a contractible loop, 1 is not an eigenvalue of the differential $(h_\ast)_x$.
\end{defi}
The followings proposition summaries the properties of spectral invariants which we need to show our result (\cite{Oh}, \cite{U1}).

\begin{prop}[\cite{Oh}, \cite{U1}]\label{spec inv}
Spectral invariants has the following properties.

\begin{description}
\item[(1)\textit{Non-degenerate spectrality}] $c(a,H)\in\operatorname{Spec}(H)$ for every non-degenerate $H\in{C}^\infty(M{\times}[0,1])$.
\item[(2)\textit{Hamiltonian shift property}]$c(a,H+{\lambda}(t))=c(a,H)+\int_0^1{\lambda}(t)dt$.
\item[(3)\textit{Monotonicity property}]If $H_1 \leq {H_2}$, then $c(a,H_1) \leq c(a,H_2)$.
\item[(4)\textit{Lipschitz property}]The map $H{\mapsto}c(a,H)$ is Lipschitz on $C^\infty(M\times[0,1])$ with respect to the $C^0$-norm.
\item[(5)\textit{Symplectic invariance}]$c(a,f^{\ast}{H})=c(a,H)$ for any $f \in \mathrm{Symp}_0(M,\omega)$ and any $H \in C^{\infty}(M\times[0,1])$.
\item[(6)\textit{Homotopy invariance}]$c(a,H_1)=c(a,H_2)$ for any normalized $H_1$ and $H_2$ generating the same $h\in\widetilde{\mathrm{Ham}}(M)$. Thus one can define $c(a,\cdot){\colon}\widetilde{\mathrm{Ham}}(M)\to\mathbb{R}$ by $c(a,h)=c(a,H)$, where $H$ is a normalized Hamiltonian function generating $h$.
\item[(7)\textit{Triangle inequality}]$c(a\ast{b},fg) \leq c(a,f)+c(b,g)$ for elements $f$ and $g \in \widetilde{\mathrm{Ham}}(M,\omega)$, where $\ast$ denotes the quantum product.
\end{description}
\end{prop}

\subsection{Heaviness and superheaviness}

M. Entov and L. Polterovich (\cite{EP09}) defined \textit{heaviness} and \textit{superheaviness} of closed subsets in closed symplectic manifolds and gave examples of non-displaceable subsets and strongly non-displaceable subsets.

For an idempotent $a$ of the quantum homology $QH_\ast(M,\omega)$, define the functional $\zeta_a{\colon}C^\infty(M)\to\mathbb{R}$ by
\[\zeta_a(H)=\lim_{l\to\infty}\frac{c(a,lH)}{l},\]
where $c(a,H)$ is the spectral invariant (\cite{Oh}, see Section \ref{introduction to spectral invariant}).

\begin{defi}[\cite{EP09}]
Let $(M,\omega)$ be a $2n$-dimensional closed symplectic manifold and $a$ be an idempotent of the quantum homology $QH_\ast(M,\omega)$.
A closed subset $X$ of ${M}$ is said to be $a$-\textit{heavy} if
\[\zeta_a(H)\geq\inf_XH\text{ for any }H \in C^\infty(M),\]
and is said to be $a$-\textit{superheavy} if
\[\zeta_a(H) \leq \sup_XH\text{ for any }H \in C^\infty(M).\]
A closed subset $X$ of ${M}$ is called heavy (respectively, superheavy) if $X$ is $a$-heavy (respectively, $a$-superheavy) for some idempotent $a$ of $QH_\ast (M,\omega)$.
\end{defi}

\begin{exam}\label{torus and cp^n}
Let $(\mathbb{C}P^n,\omega_{FS})$ be the complex projective space with the Fubini-Study form. The Clifford torus $C=\{[z_0:{\cdots}:z_n]\in\mathbb{C}P^n; |z_0|={\cdots}=|z_n|\}\subset\mathbb{C}P^n$ is a $[\mathbb{C}P^n]$-superheavy subset of $(\mathbb{C}P^n,\omega_{FS})$, hence they are strongly non-displaceable (\cite{BEP} Lemma 5.1, \cite{EP09} Theorem 1.8).
\end{exam}

For a closed oriented manifold $M$, we denote its fundamental class by $[M]$. It is known that $[M]$ is an idempotent of $QH_\ast(M,\omega)$.

\begin{thm}[A part of Theorem 1.4 of \cite{EP09}]\label{non-displaceability}

For a non-trivial idempotent $a$ of ${QH}_\ast(M,\omega)$, the followings hold.
\begin{itemize}
\item[(1)]Every $a$-superheavy subset is $a$-heavy.
\item[(2)]Every $a$-heavy subset is non-displaceable.
\item[(3)]Every $[M]$-superheavy subset is strongly non-displaceable.
\end{itemize}
\end{thm}

\begin{defi}
Let $(M,\omega)$ be a $2n$-dimensional closed symplectic manifold. Take an idempotent $a$ of the quantum homology $QH_\ast(M,\omega)$.
An open subset $U$ of $M$ is said to satisfy \textit{the bounded spectrum condition} (with respect to $a$) if there exists a constant $E>0$ such that
\[|c(a,F)| \leq E\]
for any Hamiltonian function $F\in{C}^\infty(U\times[0,1])$ with support in $U\times[0,1]$.
\end{defi}

Open subsets satisfying the bounded spectrum condition play an essential role in the present paper.

\begin{exam}
A stably displaceable subset of a closed symplectic manifold satisfies the bounded spectrum condition with respect to any idempotent $a$ (\cite{Se} Lemma 4.1). In particular, a displaceable subset of a closed symplectic manifold satisfies the bounded spectrum condition with respect to any idempotent $a$ (\cite{U2} Proposition 3.1).
\end{exam}

\section{Main proposition}

Open subsets $U$ with volume greater than the half of that of $M$ cannot be displaced but some of them satisfy the bounded spectrum condition for non simply connected symplectic manifold.

The main result of this paper is the following theorem.
\begin{prop}\label{spectral main theorem}
Let $(M,\omega)$ be a closed symplectic manifold.
Let $\alpha$ be a nontrivial free homotopy class of free loops on $M$; $\alpha\in [\mathbb{S}^1,M]$, $\alpha\neq 0$. Let $U$ be an open subset of $M$.
Assume that there exists a Hamiltonian function $H\in{C}^\infty(M\times[0,1])$ which satisfies the followings:
\begin{itemize}
\item[(1)]$h_1|_U=\mathrm{id}_U$,
\item[(2)]for any $x\in{U}$, the free loop $\gamma^x{\colon}\mathbb{S}^1\to{M}$ defined by $\gamma^x(t)=h_t(x)$
belongs to $\alpha$, and
\item[(3)]$\alpha\notin{i}_\ast([\mathbb{S}^1,U])$.
\end{itemize}
Here $i{\colon}U{\to}M$ is the inclusion map and $\{h_t\}_{t \in [0,1]}$ is the Hamiltonian path generated by $H$.
Then $U$ satisfies the bounded spectrum condition with respect to any idempotent $a$ of ${QH}_\ast(M,\omega)$.
\end{prop}

The proof of Proposition \ref{spectral main theorem} is based on the idea of K. Irie in the proof of Theorem 2.4 of \cite{Ir}.

\begin{proof}
Fix a Hamiltonian function $F\in{C}^\infty_c(U\times[0,1])$.
To use the non-degenerate spectral property, we approximate $H$ by non-degenerate Hamiltonian functions.
Take a sequence of non-degenerate Hamiltonian functions $H_n$ which converges to $H$ in the $C^2$-norm.
We denote by $\{h_{n,t}\}_{t\in[0,1]}$ the Hamiltonian path generated by $H_n$ and denote by $\gamma^x_n$ the path defined by $\gamma^x_n(t)=h_{n,t}(x)$.
For a large enough number $n\in\mathbb{N}$ and $x \in \bigcup_{t \in [0,1]}\mathrm{supp}(F_t)$, there exists a path $\beta_n^x{\colon}[0,1]\to{U}$ in $U$ such that $\beta_n^x(0)=h_{n,1}(x)$ and $\beta_n^x(1)=x$ and the composed path $\gamma^x_n \,\sharp\, \beta_n^x$ represents $\alpha\in[\mathbb{S}^1,M]$.

Choose a smooth function $\chi{\colon}[0,\frac{1}{2}]\to[0,1]$ and a positive real number $\epsilon\in(0,\frac{1}{4})$ satisfying the following:

\begin{itemize}
\item $\chi^\prime(t)\geq0$ for any $t\in{[0,\frac{1}{2}]}$, and
\item $\chi(t)=0$ for any $t\in{[0,{\epsilon}]}$, and $\chi(t)=1$ for any $t\in{[\frac{1}{2}-{\epsilon},\frac{1}{2}]}$.
\end{itemize}

For $u \in [0,1]$ we define the new Hamiltonian function $L_{n}^u \in C^\infty(M\times[0,1])$ as follows:
\begin{equation*}
L_{n}^u(x,t)=
\begin{cases}
\chi^\prime(t)H_n(x,\chi(t)) & \text{when }t\in[0,\frac{1}{2}] \\
u\chi^\prime(t-\frac{1}{2})F(x,\chi(t-\frac{1}{2}))& \text{when }t\in[\frac{1}{2},1].
\end{cases}
\end{equation*}
Since $\chi$ is constant on neighborhoods of $0$ and $\frac{1}{2}$, $L_{n}^u$ is a smooth Hamiltonian function.

We claim that $\operatorname{Spec}(L_{n}^u) \subset \operatorname{Spec}(H_n)$ for an large enough number $n\in\mathbb{N}$ and any $u\in[0,1]$. We denote by ${\{l_{n,t}^u\}}_{t \in [0,1]}$ the Hamiltonian path generated by $L_{n}^u$.
Let $x\in{M}$ be a fixed point of $l_{n,1}^u$ whose orbit $\lambda_n^{u,x}$ defined by $\lambda_n^{u,x}(t)=l_{n,t}^u(x)$ is contractible. If $x\notin{\bigcup_{t\in[0,1]}\mathrm{supp}(F_t)}$, $x$ is also a fixed point of $h_{1}$ and $\lambda^{u,x}_n(t)$ coincides with $\gamma^x_n$ up to parameter change. Hence $\gamma^x_n$ is contractible.  Since $\int_0^1H_n(h_t(x),t)dt=\int_0^1L^u_n(l_t^u(x),t)dt$, the element of $\operatorname{Spec}L^u_n$ given by the fixed point of $x$ belong to $\operatorname{Spec}(H_n)$. If $x \in \bigcup_{t\in[0,1]}\mathrm{supp}(F_t)$, since $n$ is assumed to be large enough, there exists a path $\beta_n^x$ in $U$ such that $\beta_n^x(0)=h_1^n(x)$ and $\beta_n^x(1)=x$ and $\gamma^x_n \,\sharp\, \beta_n^x$ represents $\alpha\in[\mathbb{S}^1,M]$. Since $\bigcup_{t \in [0,1]}\mathrm{supp}(F_t)\subset U$, the free loop $(\bar{\beta}_n^x\,\sharp\,\bar{\gamma}_n^x)\,\sharp\lambda_n^{u,x}$ is homotopic to a free loop in $U$. If $\lambda_n^{u,x}$ is contractible, $\bar{\beta}_n^x\,\sharp\,\bar{\gamma}_n^x$ is also homotopic to a free loop in $U$ and this contradicts $\alpha\notin{i}_\ast([\mathbb{S}^1,U])$. Hence $\lambda_n^{u,x}$ is not contractible. Thus $\operatorname{Spec}(L_n^u) \subset \operatorname{Spec}(H_n)$ holds.
Since $L_n^0$ and $H_n$ generate the same element of $\widetilde{\mathrm{Ham}}(M,\omega)$, the homotopy invariance implies 
\[c(a,L_n^0-\int_ML_n^0\omega^n)=c(a,H_n-\int_MH_n\omega^n).\]
By the Hamiltonian shift property and $\int_0^1\int_ML_n^0\omega^ndt=\int_0^1\int_MH_n\omega^ndt$, \[c(a,L_n^0)=c(a,L_n^0-\int_ML_n^0\omega^n)+\int_0^1\int_ML_n^0\omega^ndt=c(a,H_n-\int_MH_n\omega^n)+\int_0^1\int_MH_n\omega^ndt=c(a,H_n).\]

The Lipschitz property asserts that $c(a,L_n^u)$ depends continuously on $u$. Since $\operatorname{Spec}(H_n)$ is a measure-zero set, the non-degenerate spectrality implies that $c(a,L_n^u)$ is a constant function of $u$. Hence $c(a,L_n^u)=c(a,H_n)$ for any $u \in [0,1]$.

Since $L_n^1$ and $F{\sharp}H_n$ generates the same element of $\widetilde{\mathrm{Ham}}(M,\omega)$, by a computation as above, $c(a,F{\sharp}H_n)=c(a,L^1_n)$. Thus $c(a,F{\sharp}H_n)=c(a,L^1_n)=c(a,H_n)$. Then the triangle inequality implies
\begin{align*}
c(a,F)& \leq c(a,F{\sharp}H_n)+c(a,\bar{H_n})\\
      &=c(a,L_n^1)+c(a,\bar{H_n})\\
      &=c(a,H_n)+c(a,\bar{H_n}),\text{ and}
\end{align*}
\begin{align*}
c(a,F)& \geq c(a,F{\sharp}H_n)-c(a,H_n)\\
      &=c(a,H_n)+c(a,H_n)=0.
\end{align*}
Since Lipschitz properties implies $\lim_{n\to\infty}c(a,H_n)=c(a,H)$ and $\lim_{n\to\infty}c(a,\bar{H_n})=c(a,\bar{H})$, we have
\[0\leq c(a,F) \leq c(a,H)+c(a,\bar{H}).\]
\end{proof}

\section{the bounded spectrum condition and an $a$-stem}\label{qs lemma}

We need the argument in this section to prove the superheaviness of Theorem \ref{torus}.

\subsection{An $a$-stem}

\begin{defi}
An open subset $U$ of $M$ is said to be $a$-null
if for $G\in C^\infty(U)$,
\[\zeta_a(G)=0.\]
An open subset $U$ of $M$ is said to be strongly $a$-null
if for $F\in C^\infty(M)$ and $G\in C^\infty(U)$ with $\{F,G\}=0$,
\[\zeta_a(F+G)=\zeta_a(F).\]
A subset $X$ of $M$ is said to be (strongly) $a$-null if there exists a (strongly) $a$-null open neighborhood $U$ of $X$.
\end{defi}

$a$-nullity is defined by Monzner-Vichery-Zapolsky \cite{MVZ}.
If a subset $X$ of $M$ is strongly $a$-null, $X$ is $a$-null.

\begin{prop}\label{bsc is scc}
Let $(M,\omega)$ be a $2n$-dimensional closed symplectic manifold. For an idempotent $a$ of $QH_\ast(M,\omega)$, if an open subset $U$ of $M$ satisfies the bounded spectrum condition with respect to $a$, then $U$ is strongly $a$-null.
\end{prop}

M. Entov and L. Polterovich defined stems to give examples of superheavy subsets.
We define $a$-stems which generalizes a little the notion of stems and they exhibits $a$-superheaviness.

We generalize the argument of Entov and Polterovich as follows.

\begin{defi}\label{generalized stem}
Let $\mathbb{A}$ be a finite-dimensional Poisson-commutative subspace of $C^\infty(M)$ and $\Phi\colon M\to\mathbb{A}^\ast$ be the moment map defined by $\langle \Phi(x),F \rangle =F(x)$. Let $a$ be a non-trivial idempotent of $QH_\ast(M,\omega)$. A non-empty fiber $\Phi^{-1}(p)$, $p \in \mathbb{A}^\ast$ is called \textit{an $a$-stem} of $\mathbb{A}$ if all non-empty fibers $\Phi^{-1}(q)$ with $q \neq p$ is strongly $a$-null. If a subset of $M$ is an $a$-stem of a finite-dimensional Poisson-commutative subspace of $C^\infty(M)$, it is called just \textit{an $a$-stem}.
\end{defi}

The following theorem is proved in Subsection \ref{generalize stem}. The proof of this theorem is almost same as the proof of Theorem 1.8 of \cite{EP09}, but we prove it for the convenience of the reader.

\begin{thm}\label{generalized stem sheavy}
For every idempotent $a$ of $QH_\ast (M,\omega)$, every $a$-stem is an $a$-superheavy subset.
\end{thm}

\subsection{Asymptotic spectral invariants and Proof of Proposition \ref{bsc is scc}}\label{qs lemmas}

\begin{defi}
Let $(M,\omega)$ be a $2n$-dimensional closed symplectic manifold.
For an idempotent $a$ of ${QH}_\ast(M,\omega)$, we construct \textit{asymptotic spectral invariant} $\mu_a{\colon}\widetilde{\mathrm{Ham}}(M,\omega)\to\mathbb{R}$ as follows (\cite{EP03}, \cite{EP06}):
\[\mu_a(f)=-\mathrm{Vol}(M,\omega )\cdot\lim_{l\to\infty}\frac{c(a,f^l)}{l},\]
where  $\mathrm{Vol}(M,\omega )=\int_M{\omega}^n$ is the volume of $(M,\omega)$.
\end{defi}

For any $H\in C^\infty(M)$, by the definition of $\zeta_a$ and $\mu_a$,
\[\zeta_a(H)=-\mathrm{Vol}(M,\omega )^{-1}(\int_MH\omega^n-\mu_a(h)),\]
where $h$ is the Hamiltonian diffeomorphism generated by $H$.

$\mu_a$ is known to be an \textit{Entov-Polterovich pre-quasimorphism} (\cite{EP06}, \cite{FOOO}).
Entov-Polterovich \cite{EP03} and Ostrover \cite{Os06} gave several conditions on $a\in{QH}_\ast(M,\omega)$ under which $\mu_a$ is a \textit{Calabi quasi-morphism}.
We omit the definitions of an Entov-Polterovich pre-quasimorphism and a Calabi quasi-morphism.

\noindent
\begin{defi}[{\cite{Ba}}]\label{Fragmentation}

Let $(M,{\omega})$ be a $2n$-dimensional closed symplectic manifold.
Given an open set $U$ of ${M}$, we denote by $\widetilde{\mathrm{Ham}}_U(M,\omega)$ the set of elements of $\widetilde{\mathrm{Ham}}(M,\omega)$ which is generated by Hamiltonian functions with support in $U$. By Banyaga's fragmentation lemma \cite{Ba}, each $\phi \in \widetilde{\mathrm{Ham}}(M,\omega)$ can be represented as a product of elements of the form $\psi\theta{\psi}^{-1}$ with $\theta \in\widetilde{\mathrm{Ham}}_U(M,\omega)$, $\psi \in \widetilde{\mathrm{Ham}}(M,\omega)$.
Denote by $||{\phi}||_U$ the minimal number of factors in such a product.
This is called the \textit{fragmentation norm}.

\end{defi}

\noindent
\begin{prop}[Controlled quasi-additivity and Calabi property]\label{cq}

Let $(M,\omega)$ be a closed symplectic manifold and $a$ be an idempotent of ${QH_\ast(M,\omega)}$.
For an open subset $U$ of $M$ satisfying the bounded spectrum condition with respect to $a$, the followings hold.
\begin{itemize}
\item[(1)](Controlled quasi-additivity) There exists a constant $K$ depending only on $U$ such that \[|\mu_a(fg)-\mu_a(f)-\mu_a(g)|<K{\,}\min\{||f||_U,||g||_U\}\text{ for any }f,g\in\widetilde{\mathrm{Ham}}(M,\omega)\]
\item[(2)](Calabi property) For a Hamiltonian diffeomorphism $h$ generated by a Hamiltonian function $H$ with support in $U$, 
\[\mu_a(h)=\int_0^1\int_MH\omega^ndt.\]
\end{itemize}
\end{prop}

\noindent
\begin{lem}\label{fragm=1}
For an idempotent $a$ of $QH_\ast(M,\omega)$,
there exists a constant $E>0$ such that for any $f$ and $g\in\widetilde{\mathrm{Ham}}(M,\omega)$ with $||g||_U=1$,

\[|c(a,fg)-c(a,f)|<E.\]
\end{lem}

\begin{proof}

Since $U$ satisfies the bounded spectrum condition, there is a positive real number $E$ such that
\[|c(a,h)|\leq E\text{ for every }h\in\widetilde{\mathrm{Ham}}_U(M,\omega).\]
By the triangle inequality, $c(a,fg){\leq}c(a,f)+c(a,g)$ and $c(a,f){\leq}c(a,fg)+c(a,g^{-1})$.
Assumptions $||g||_U=1$ and the symplectic invariance imply that $-E{\leq}c(a,g){\leq}E$ and $-E{\leq}c(a,g^{-1}){\leq}E$.
Thus
\[-E{\leq}c(a,g^{-1}){\leq}c(a,fg)-c(a,f){\leq}c(a,g){\leq}E.\]
\end{proof}

\begin{proof}
[Proof of Proposition \ref{cq}]

First, we prove (1). Our argument is based on the proof of Theorem 7.1 in \cite{EP06}.
Take any element $f\in \widetilde{\mathrm{Ham}}(M,\omega)$ and represent it as $f=f_1{\dots}f_m$ with $||f_i||_U=1$ for all $i$.
We claim that
\[|\mu_a(fg)-\mu_a(f)-\mu_a(g)|{\leq}2E(2m-1)\mathrm{Vol}(M,\omega).\]
We prove the claim by induction on $m$.

For $m=1$, since $(fg)^k=\Pi_{i=0}^{i=k-1}(g^ifg^{-i}){\cdot}g^k$ and $||f||_U=1$, by Lemma \ref{fragm=1},
\[|c(a,(fg)^k)-\sum_{i=1}^kc(a,g^ifg^{-i})-c(a,g^k)| \leq Ek.\]
By the symplectic invariance,
\[|c(a,(fg)^k)-kc(a,f)-c(a,g^k)| \leq Ek.\]
Combining this with the inequality
\[|c(a,f^k)-kc(a,f)| \leq E(k-1),\]
we attain
\[|c(a,(fg)^k)-c(a,f^k)-c(a,g^k)|\leq E(2k-1)\]
Dividing by $k$ and passing to the limit as $k\to\infty$, we obtain
\begin{align*}
&|\mu_a(fg)-\mu_a(f)-\mu_a(g)|\\
& =\lim_{k\to\infty}\mathrm{Vol}(M,\omega) k^{-1}|c(a,(fg)^k)-c(a,f^k)-c(a,g^k)|\\
& \leq  2E\cdot\mathrm{Vol}(M,\omega).
\end{align*}

\noindent
Thus (1) is true for $m=1$.

Assume that the claim is true for $||f||_U=m$. For $f$ such that $||f||_U=m+1$, we decompose $f=f_mf_1$, where $||f_m||_U=m$ and $||f_1||=1$.
Then
\[|\mu_a(f_mf_1g)-\mu_a(f_m)-\mu_a(f_1g)|{\leq}2E(2m-1)\mathrm{Vol}(M,\omega),\]
\[|\mu_a(f_1g)-\mu_a(f_1)-\mu_a(g)|{\leq}2E\mathrm{Vol}(M,\omega)\]
and
\[|\mu_a(f_1)+\mu_a(f_m)-\mu_a(f_mf_1)|{\leq}2E\mathrm{Vol}(M,\omega).\]
By these inequalities, 
\[|\mu_a(fg)-\mu_a(f)-\mu_a(g)|{\leq}2E(2m+1)\mathrm{Vol}(M,\omega).\]
This completes the proof of the controlled quasi-additivity (1).

Secondly, we prove (2).
Set $\lambda(t)=\dfrac{\int_MH_t\omega^n}{\mathrm{Vol}(M,\omega )}$.
Since the normalization of $kH_t$ is $kH_t-k\lambda(t)$, the definition of spectral invariants for elements of $\widetilde{\mathrm{Ham}}(M,\omega)$ and the Hamiltonian shift property imply,
\begin{align*}
\mu_a(h)=&-\mathrm{Vol}(M,\omega )\lim_{k\to\infty}\frac{c(a,h^k)}{k}\\
=&-\mathrm{Vol}(M,\omega )\lim_{k\to\infty}\frac{c(a,kH-k\lambda)}{k}\\
=&-\mathrm{Vol}(M,\omega )\lim_{k\to\infty}\frac{c(a,kH)}{k}+\mathrm{Vol}(M,\omega )\int_0^1\lambda(t)dt\\
=&-\mathrm{Vol}(M,\omega )\lim_{k\to\infty}\frac{c(a,kH)}{k}+\int_0^1\int_MH\omega^ndt.
\end{align*}
The bounded spectrum condition implies that there exists a constant $E>0$ such that $|c(a,kH)|<E$ for any $k$. Thus
\[\mathrm{Vol}(M,\omega )\cdot\lim_{k\to\infty}\frac{c(a,kH)}{k}=0.\]
Hence Proposition \ref{cq}(2) is proved.
\end{proof}

\begin{lem}\label{kahou2}

Under the assumption of Proposition \ref{cq}, for the Hamiltonian diffeomorphism $f$ generated by a Hamiltonian function with support in $U$ and a Hamiltonian diffeomorphism $g$ satisfying $fg=gf$
\[\mu_a(fg)=\mu_a(f)+\mu_a(g).\]
\end{lem}

\begin{proof}
Note that $||{f}^n||_U\leq1$ for all $n\in\mathbb{Z}_{\geq{0}}$. We use the semi-homogeneity and the controlled quasi-additivity of $\mu_a$.
For any $n\in\mathbb{Z}$,
\begin{align*}
n\lvert\mu_a(fg)-\mu_a(f)-\mu_a(g)\rvert
&=\lvert\mu_a((fg)^n)-\mu_a(f^n)-\mu_a(g^n)\rvert\\
&\leq\lvert\mu_a(f^ng^n)-\mu_a(f^n)-\mu_a(g^n)\rvert\\
&<K{\,}\min\{||f^n||_U,||g^n||_U\}\\
&<K,\\
\end{align*}
where $K$ is the constant in the definition of Proposition \ref{cq} (1).
Therefore $\mu_a(fg)=\mu_a(f)+\mu_a(g)$.
\end{proof}

\begin{proof}[Proof of Proposition \ref{bsc is scc}]

Assume that an open subset $U$ of $M$ satisfies the bounded spectrum condition with respect to an idempotent $a\in QH_\ast (M,\omega)$.
Take Hamiltonian functions $F\in C^\infty(M)$ and $G\in C^\infty(U)$ such that $\{F,G\}=0$.
Since $\{F,G\}=0$, a function $F+G$ generates the Hamiltonian diffeomorphism $fg=gf$ where $f$ and $g$ are the Hamiltonian diffeomorphisms generated by $F$ and $G$, respectively.

By Lemma \ref{kahou2} and Proposition \ref{cq}(2),
\[ \zeta_a(F+G)=\mathrm{Vol}(M,\omega )^{-1} (\int_M(F+G)\omega^n-\mu_a(fg))\]
\[=\mathrm{Vol}(M,\omega )^{-1} (\int_MF\omega^n+\int_MG\omega^n-\mu_a(f)-\mu_a(g))=\zeta_a(F).\]
Thus Proposition \ref{bsc is scc} holds.
\end{proof}

\subsection{Proof of Theorem \ref{generalized stem sheavy}}\label{generalize stem}

First we prove the following elementary lemma.

\begin{lem}\label{thin}

Let $X$ be a closed subset of a closed symplectic manifold $(M,\omega)$. Take an idempotent $a$ of $QH_\ast(M,\omega)$.
If $\bar{U}$ is $a$-superheavy for any open neighborhood $U$ of $X$, then $X$ is also $a$-superheavy.

\end{lem}

\begin{proof}

Fix a Hamiltonian function $F\in{C}^\infty(M)$.
For any ${\epsilon}>0$, there exists an open neighborhood $U$ of $X$ such that $\sup_{\bar{U}}F \leq \sup_XF+\epsilon$.
Since the assumption implies that $\bar{U}$ is $a$-superheavy, $\zeta_a(F) \leq \sup_{\bar{U}}F \leq \sup_XF+\epsilon$.
Since ${\epsilon}$ is any positive real number, $\zeta_a(F) \leq \sup_XF$.
Thus $X$ is $a$-superheavy.

\end{proof}

To prove Theorem \ref{generalized stem sheavy}, we use the following proposition.

\begin{prop}[{\cite{EP09} Proposition 4.1}]\label{EP fundamental property}
A closed subset $X$ of ${M}$ is heavy if and only if, for any $H\in{C}^\infty(M)$ with $H|_X=0$ and $H\leq{0}$, one has ${\zeta_a}(H)=0$. A closed subset $X$ of ${M}$ is $a$-superheavy if and only if, for any $H\in{C}^\infty(M)$ with $H|_X=0$ and $H\geq{0}$, one has ${\zeta_a}(H)=0$.
\end{prop}

\begin{proof}[Proof of Theorem \ref{generalized stem sheavy}]

Let $\mathbb{A}$ be a finite-dimensional Poisson-commutative subspace of $C^\infty(M)$ and $\Phi\colon M\to\mathbb{A}^\ast$ be its moment map and let $X=\Phi^{-1}(p)$ be an $a$-stem of $\mathbb{A}$.
Take an open neighborhood $V$ of $X$.  Take a Hamiltonian function $\hat{H}$ such that $\hat{H}|_V=0$ and $\hat{H}\geq 0$. By Lemma \ref{thin} and Proposition \ref{EP fundamental property}, it is sufficient to prove $\zeta_a(\hat{H})=0$.

One can find an open neighborhood $U$ of $p$ and a function $H\in C^\infty(\mathbb{A}^\ast)$ with $H\geq 0$ and $H|_U=0$ such that $\hat{H}\leq \Phi^\ast H$.
Then by monotonicity
\[0\leq \zeta_a(\hat{H})\leq \zeta_a(\Phi^\ast H).\]
Thus it is sufficient to prove $\zeta_a(\Phi^\ast H)=0$.

By the definition of an $a$-stem, we can choose an open covering $\{U_0,U_1,\cdots,U_N\}$ of $\Phi(M)$ such that $U_0=U$, and all $\Phi^{-1}(U_i)$ is strongly $a$-null for $i \geq 1$. Let $\rho_i{\colon}\Phi(M)\to\mathbb{R}$, $i=0,\cdots,N$, be a partition of unity subordinated to the covering $\{U_i\}$. Then
\[\Phi^\ast H=\sum_{i=0}^N\Phi^\ast H\,\rho_i=\Phi^\ast H\,\rho_0+\sum_{i=1}^N\Phi^\ast H\,\rho_i.\]
Since all $\mathrm{supp}(\Phi^\ast H\,\rho_i)$ is strongly $a$-null for $i\geq 1$ and $H|_U=0$,
\[\zeta_a(\Phi^\ast H)=\zeta_a(\sum_{i=0}^N\Phi^\ast H\, \rho_i)=\zeta_a(\Phi^\ast H\,\rho_0)=0.\]
\end{proof}

\section{Proof of Theorem \ref {torus} and Corollary \ref{p-torus}} \label{exam}

\begin{proof}[Proof of Theorem \ref{torus}]

Consider a moment map $\Phi\in C^\infty(\mathbb{T}^2)$ such that $\Phi(x)=0$ if $x\in M \cup L$ and $\Phi(x)>0$ if $x\notin M \cup L$. Take a real number $\epsilon\neq 0$.
Then there exist a positive number $\delta$ and an open neighborhood $U$ of $\Phi^{-1}(\epsilon)$ such that $U\subset (\delta,1-\delta)\times(\delta,1-\delta)$.
Consider a Hamiltonian function $H\in{C}^\infty(\mathbb{T}^2\times[0,1])$ such that $H((p,q),t)=p$ for any $p \in [\delta,1-\delta]$.

Define the free loop $\gamma\colon \mathbb{S}^1\to \mathbb{T}^2$ by $\gamma(t)=(0,t)$. Let $\alpha\in [\mathbb{S}^1,\mathbb{T}^2]$ be the homotopy class of free loops represented by $\gamma$.
Then $\alpha$, $U$ and $H$ satisfy the assumptions of Proposition \ref{spectral main theorem}, hence $U$ satisfies the bounded spectrum condition with respect to any idempotent $a\in QH_\ast(\mathbb{T}^2,\omega_{\mathbb{T}^2})$ and is strongly $a$-null by Theorem \ref{bsc is scc}. Thus $M \cup L$ is an $a$-stem, hence it is $a$-superheavy.
\end{proof}

Though the above example cannot be displaced by homeomorphisms, it gives rise to a nontrivial non-displaceable example by using the following theorem.

\noindent
\begin{thm}[{\cite{EP09} Theorem 1.7}]\label{product of heavy}

Let $(M_1,{\omega}_1)$ and $(M_2,{\omega}_2)$ be closed symplectic manifolds. Take non-zero idempotents $a_1$, $a_2$ of $QH_\ast(M_1)$, $QH_\ast(M_1)$, respectively.
Assume that for $i=1,2$, $X_i$ be an $a_i$-heavy (respectively, $a_i$-superheavy) subset.
Then the product $X_1{\times}X_2$ is $a_1{\otimes}a_2$-heavy (respectively, $a_1{\otimes}a_2$-superheavy) subset of $(M_1{\times}M_2,\omega_1 \oplus \omega_2)$ with respect to the idempotent  of $QH_\ast(M_1{\times}M_2)$.
\end{thm}

\begin{proof}[Proof of Corollary \ref{p-torus}]
By Example \ref{torus and cp^n}, Theorem \ref{torus} and Theorem \ref{product of heavy}, $C{\times}(M\cup{L})$ is $[\mathbb{C}P^n\times \mathbb{T}^2]$-superheavy subset of $(\mathbb{C}P^n\times \mathbb{T}^2,\omega_{FS}\oplus \omega_{\mathbb{T}^2})$. Thus Theorem \ref{non-displaceability} implies that there exists no symplectomorphism $f$ such that $C{\times}(M\cup{L})\cap f(C{\times}(M\cup{L}))=\emptyset$.
\end{proof}

\section{Generalization}\label{general}

In this section, we give a generalization of Theorem \ref{torus}.

\begin{thm}\label{torus2}
Let $(\Sigma_g,\omega)$ be a closed Riemannian surface with a symplectic (area) form $\omega$ and $e^0\cup e^1_1\cup\cdots\cup e^1_{2g}\cup e^2$ its CW-decomposition.
Then $e^1_1\cup\cdots\cup e^1_{2g}$ is a ``$[\Sigma_g]$-superheavy'' subset of $(\sigma_g,\omega)$,
\end{thm}

\begin{rem}
Humili$\acute{e}$re, Le Roux and Seyfaddini \cite{HLS} found the similar theorem to Theorem \ref{torus2} independently (See also \cite{Is}).
\end{rem}

\begin{proof}[Proof of Theorem \ref{torus2}]
By cutting $\Sigma_g$  open along $e^1_1\cup\cdots\cup e^1_{2g}$, we construct $4g$-gon $\tilde{\Sigma}_g$ and the natural quatient map $\pi\colon\tilde{\Sigma}_g\to\Sigma_g$.
We mark all sides of  $\tilde{\Sigma}_g$ by $e^u_1,\ldots,e^u_{2g}$ and $e^l_1,\ldots,e^l_{2g}$ such that $\pi(e^u_i)=e^1_i$ and $\pi(e^l_i)=e^1_i$.

Put $A=\int_{\Sigma_g}\omega$ and let $S_g$ be a square in $\mathbb{R}^2$ defined by $S_g=[0,1]\times[0,A]$.
Let $s^u$ and $s^l$ denote the sides $[0,1]\times\{A\}$ and $[0,1]\times\{0\}$ of $S_g$, respectively.
Then we can take  an area-preserving diffeomorphism $f\colon S_g\to\tilde\Sigma$ such that $f(s^u)=e^u_1$, $f(s^l)=e^l_1$ and $\pi(f((t,0)))=\pi(f((t,A)))$ for any $t\in[0,1]$.

Consider a moment map $\Phi\colon \Sigma_g\to\mathbb{R}$ such that $\Phi(x)=0$ if $x\in e^1_1\cup\cdots\cup e^1_{2g}$ and $\Phi(x)>0$ if $x\notin e^1_1\cup\cdots\cup e^1_{2g}$. Take a real number $\epsilon\neq 0$.
Then there exist a positive number $\delta$ and an open neighborhood $U$ of $\Phi^{-1}(\epsilon)$ such that $U\subset (\delta,1-\delta)\times(\delta,1-\delta)$.

Consider a function $\hat{H}\colon S_g\to\mathbb{R}$ such that  $\hat{H}((p,q))=Ap$ for any $p \in [\delta,1-\delta]$ and $\hat{H}((p,q))=0$ for any $p \in [0,\frac{\delta}{2}]\cup[1-\frac{\delta}{2},1]$.
Since $\pi(f((t,0)))=\pi(f((t,A)))$ for any $t\in[0,1]$ and $\hat{H}((p,q))=0$ for any $p \in [0,\frac{\delta}{2}]\cup[1-\frac{\delta}{2},1]$, there exists a Hamiltonian function $H\colon \Sigma_g\to\mathbb{R}$ such that  $\hat{H}=H\circ\pi\circ f$.

Define  the path $\hat\gamma\colon \mathbb{S}^1\to S_g$ by $\gamma(t)=(0,At)$ and the free loop $\gamma\colon \mathbb{S}^1\to \Sigma_g$ by $\gamma=\pi\circ f\circ\hat\gamma$. Let $\alpha\in [\mathbb{S}^1,\Sigma_g]$ be the homotopy class of free loops represented by $\gamma$.
Then $\alpha$, $U$ and $H$ satisfy the assumptions of Proposition \ref{spectral main theorem}, hence $U$ satisfies the bounded spectrum condition with respect to any idempotent $a\in QH_\ast(\Sigma_g,\omega)$ and is strongly $a$-null by Theorem \ref{bsc is scc}. Thus $M \cup L$ is an $a$-stem, hence it is $a$-superheavy.
\end{proof}


\begin{thebibliography}{ABCD}
 
 \bibitem[Ba]{Ba}
{A. Banyaga},
{\em Sur la structure du groupe de des diff$\acute{e}$omorphismes qui preservent une forme symplectique},
{Comm. Math. Helv.},
 \textbf{53} (3) (1978), 174-227.

 \bibitem[BEP]{BEP}
{P. Biran, M. Entov and L. Polterovich},
{\em Calabi quasimorphisms for the symplectic ball},
{Commum. Contemp. Math.},
 \textbf{6} (2004), 793-802.

 \bibitem[EP03]{EP03}
{M. Entov and L. Polterovich},
{\em Calabi quasimorphism and quantum homology},
{Internat. Math. Res. Notices},
 \textbf{30} (2003), 1635-1676.

 \bibitem[EP06]{EP06}
{M. Entov and L. Polterovich},
{\em Quasi-states and symplectic intersections},
{Comment. Math. Helv.},
 \textbf{81} (1) (2006), 75-99.

 \bibitem[EP09]{EP09}
{M. Entov and L. Polterovich},
{\em Rigid subsets of symplectic manifolds},
{Comp. Math.},
 \textbf{145} (3) (2009), 773-826.
 
 \bibitem[FOOO]{FOOO}
{K. Fukaya, Y. G. Oh, H. Ohta and K. Ono},
{\em Spectral invariants with bulk, quasimorphisms and Lagrangian Floer theory},
 arXiv:1105.5123v2.

 
 \bibitem[HLS]{HLS}
{V. Humili$\acute{\mathrm{e}}$re, F. Le Roux and S. Seyfaddini},
{\em Towards a dynamical interpretation of Hamiltonian spectral invariants on surfaces},
{Geometry \& Topology},
\textbf{20}(4) (2016),2253–2334.


 \bibitem[Ir]{Ir}
{K. Irie},
{\em Hofer-Zehnder capacity and a Hamiltonian circle action with noncontractible orbits},
 arXiv:1112.5247v1.

 \bibitem[Is]{Is}
{S. Ishikawa},
{\em Spectral invariants of distance functions},
{Journal of Topology and Analysis},
\textbf{8}(4), (2016), 655-676. 

\bibitem[MVZ]{MVZ}
{A. Monzner, N. Vichery and F. Zapolsky},
{\em Partial quasimorphisms and quasistates on cotangent bundles, and symplectic homogenization},
{Journal of Modern Dynamics},
 \textbf{6}(2)  (2012), 205-249.

 \bibitem[Oh05]{Oh05}
{Y. -G. Oh},
{\em Construction of spectral invariants of Hamiltonian paths on closed symplectic manifolds},
{in $\lq\lq$The Breath of Symplectic and Poisson Geometry'', Prog. Math.},
 \textbf{232}  (2005), 525-570, Birkh$\ddot{\mathrm{a}}$user, Boston, 2005.
 
 \bibitem[Oh06]{Oh}
{Y. -G. Oh},
{\em Lectures on Floer theory and spectral invariants of Hamiltonian flows},
{Morse Theoretic Methods in Nonlinear Analysis and in Symplectic Topology, NATO Sci. Ser. II Math. Phys. Chem.y},
 \textbf{217}  (2006), 321-416.

 \bibitem[Oh09]{Oh09}
{Y. -G. Oh},
{\em Floer mini-max theory, the Cerf diagram, and the spectral invariants},
{J. Korean Math. Soc.},
 \textbf{46}  (2009), 363-447.

 
 \bibitem[Os]{Os06}
{Y. Ostrover},
{\em Calabi quasi-morphisms for some non-monotone symplectic manifolds},
{Algebr. Geom. Topol.},
 \textbf{6}  (2006), 405-434.

 \bibitem[Se]{Se}
{S. Seyfaddini},
{\em Descent and $C^0$-rigidity of spectral invariants on monotone symplectic manifolds},
{J. Topol. Anal.},
 \textbf{4}  (2012), 481-498.

\bibitem[U08]{U1}
{M. Usher},
{\em Spectral numbers in Floer theories},
{Compos. Math.},
 \textbf{144}(6)  (2008), 1581-1592.


\bibitem[U10]{U2}
{M. Usher},
{\em The sharp energy-capacity inequality},
{Commun. Comtemp. Math.},
 \textbf{12}(3)  (2010), 457-473.



\end{thebibliography}
\end{document}